\documentclass[11pt,a4paper,twoside,reqno]{amsart}
\usepackage{graphicx,calc,amscd}
\usepackage{epsfig,psfig,curves}
\usepackage{float}
\usepackage{labelfig}
\usepackage{amsfonts}
\usepackage{amssymb}
\usepackage{fontenc}
\usepackage{amsmath}

\newcommand{\fnote}[1]{\footnote{\small sharp1}}

\newcommand{\Z}{{\mathbb Z}}
\newcommand{\R}{{\mathbb R}}

\newcommand{\C}{{\mathbb C}}
\newcommand{\T}{{\mathbb T}}

\newcommand{\Dias}{\mathfrak{D}}

\newcommand{\M}{\mathcal{M}}

\newcommand{\Op}{\mathcal{O}}

\newcommand{\area}{\text{area}}
\newcommand{\dias}{\text{dias}}

\newcommand{\sys}{{\rm sys}}

\numberwithin{equation}{section}

\newtheorem{theorem}{Theorem}[section]
\newtheorem{proposition}[theorem]{Proposition}

\newtheorem{lemma}[theorem]{Lemma}

\theoremstyle{definition}

\def\d{\text{dias}\,}
\def\A{\text{area}}

\title[Local optimal diastolic inequality]{A local optimal diastolic inequality on the two-sphere}

\author[F.~Balacheff]{Florent Balacheff}

\address{Florent Balacheff - Laboratoire Paul Painlev\'e, Bt. M2, Universit\'e des Sciences et Technologies de Lille, 59655 Villeneuve d'Ascq, France}
\email{Florent.Balacheff@math.univ-lille1.fr}
\thanks{Funded by grant 20-118014/1 of the Swiss SNF}

\begin{document}
\maketitle

\begin{abstract}
Using a ramified cover of the two-sphere by the torus, we prove a local
optimal inequality between the diastole and the area on the two-sphere near a
singular metric. This singular metric, made of two equilateral triangles glued
along their boundary , has been conjectured by E. Calabi to achieve the best
ratio area over the square of the length of a shortest closed geodesic. Our diastolic inequality asserts that this conjecture is to some extent locally true. 
 
\end{abstract}

\section{Introduction}

Let $g$ be a smooth Riemannian metric on the two-sphere $S^2$. Following I. Babenko \cite{Bab97} we call {\it general systole}, or simply {\it systole} in this context, the least length of a non trivial closed geodesic and denote it by $\sys_0$. For every Riemannian metric $g$ on $S^2$,
\begin{equation}\label{eq:sys}
\A(S^2,g)\geq {1 \over 32} \sys_0(S^2,g)^2.
\end{equation}
Such an inequality was first proved by C. Croke (see \cite{Cr88}), and later improved in \cite{NR02}, \cite{Sa04} and \cite{Ro05}.
A natural question is to ask for the best constant in inequality (\ref{eq:sys}). It amounts to find the global infimum of the {\it systolic area} defined as the ratio $\A/\sys_0^2$ among smooth metrics. A flat metric with three conical singularities was conjectured by E. Calabi (see \cite{Cr88}) to be the global minimum. This metric, denoted by $g_c$, is made of two equilateral triangles of size $1$ glued
along their boundary. Its systolic area is equal to ${1\over 2\sqrt{3}}$.\\

We may wonder why such a metric has any chance to achieve the minimal systolic area.  In order to answer and then give (if needed) strength to Calabi's systolic conjecture we ask the following question : {\it Is such a conjecture locally true} ? \\

Observe that the systolic area is not a continuous function of the metric, even when the space of smooth metrics is topologized by the strong topology (for a definition of strong topology, see \cite[p.~34]{Hi76}). In order to prove local results we introduce a quantity that coincides with the systole at $g_c$ and that varies continuously with the metric $g$. This quantity is the {\it diastole} (see \cite{BS07}). The diastole of a Riemannian closed surface is defined as the value obtained by a minimax process over the space of one-cycles, see section \ref{sec:dias}. It is denoted by $\dias$. In the case of a smooth metric the diastole is realized as the length of an union of disjoint closed geodesics (see \cite[p.~468]{pit74}). Observe that the notion of diastole holds for any Riemannian closed surface with a finite number of conical singularities.
It is proved in \cite{BS07} that for every Riemannian closed  surface $(M,g)$ of genus $k$ the diastole bounds from below the area: 
$$
\area(M,g) \geq \frac{C}{k+1} \, \dias(M,g)^2. 
$$
Here $C$ is a positive constant. In the case of a two-sphere and since \linebreak $\sys_0\leq \dias$ for smooth metrics we thus (re)obtain inequality (\ref{eq:sys}) albeit with a worse constant. We call the ratio $\area/\dias^2$ the {\it diastolic area} and observe that it is continuous on the space of smooth metrics  with a finite number of conical singularities for the uniform topology on metric spaces (see \cite[chapter 7]{AZ67} for a definition of uniform topology). The infimum of the diastolic area  is not known for any surface, and we do not know if in the case of the two-sphere it is realized by $g_c$ and so equals $1/ (2\sqrt{3})$. It is proved in \cite{Ba06} that the round metric $g_0$ is a critical point of the diastolic area over the space of smooth metrics conformal to $g_0$ topologized by the strong topology. As 
$$
{\A(S^2,g_0)\over \dias(S^2,g_0)^2}={1\over \pi},
$$
the round metric is not the global infimum. So its criticity is very
interesting and suggests that it may be a local minimum. \\


In this paper we prove that the metric $g_c$ is a local minimum of the diastolic area over the space of singular metrics $C^1$-conformal to $g_c$ with respect to the $C^1$-topology. More precisely, let
$$
\M_{g_c}:=\left\{e^{2u}\cdot g_c \mid u \in
C^1(S^2,\R)\right\}
$$
be the space of metrics of class $C^1$ with three conical singularities of angle $2\pi/3$ on $S^2$ conformal to $g_c$. The space $\M_{g_c}$ is naturally in bijection with $C^1(S^2,\R^\ast_+)$ and we will call  {\it $C^1$-topology} on $\M_{g_c}$ the topology induced by the $C^1$ compact-open topology on $C^1(S^2,\R^\ast_+)$ (see \cite[p.~34]{Hi76}).

The space $\M_{g_c}$ is adapted to the local study of the diastolic area near the metric $g_c$. In fact we will observe in subsection (\ref{unif}) that, thanks to the Riemann uniformization theorem,  any smooth  metric can be written (up to an isometry) as  a metric which differs from $g_c$ by a conformal factor which lies in $C^1(S^2,\R_+)$. Especially  finding the infimum of the diastolic area over $\M_{g_c}$ amounts to find the infimum of the diastolic area over the space of smooth metrics, see proposition (\ref{prop}).

\begin{theorem}\label{th:dias-sing}
There exists an open neighborhood $\Op$ of $g_c$ in $\M_{g_c}$ with respect to the $C^1$-topology such that for all $g \in \Op$,
$$
 \A(S^2,g) \geq {1\over 2\sqrt{3}} \, \d (S^2,g)^2
$$
with equality if and only if $g=g_c$.
\end{theorem}

The proof is based on the study of a degree $3$ ramified cover of $S^2$ by the two-torus $\T^2$. 

\section{Preliminaries and definitions}

\subsection{Isodiastolic inequality for closed surfaces}\label{sec:dias}

 Recall from \cite{BS07} that the {\it diastole} of a closed Riemannian surface $(M,g)$ of class $C^1$ with a finite number of conical singularities is defined by a
minimax process over the space of one-cycles and is denoted by $\d(M,g)$. 

More precisely, we denote by ~$\mathcal{Z}_{1}(M;\Z_\ast)$ the space of one-cycles of $M$, see \cite{BS07} for a precise definition. Here $\Z_\ast$ denotes $\Z$ in the orientable case, and $\Z_2$ in the non orientable one. An isomorphism due to F.~Almgren~\cite{almgren} between $\pi_{1}(\mathcal{Z}_{1}(M;\Z_\ast),\{0\})$ and $H_{2}(M;\Z_\ast)\simeq \Z_\ast$ permits us to define the diastole over the one-cycle space as
$$
\d(M,g) := \inf_{(z_{t})} \sup_{0 \leq t \leq 1} Mass(z_t)
$$
where $(z_{t})$ runs over the families of one-cycles inducing a generator of \linebreak  $\pi_{1}(\mathcal{Z}_{1}(M;\Z_\ast),\{0\})$ and $Mass(z_t)$ represents the mass (or length) of~$z_t$. 
From a result of J.~Pitts \cite[p.~468]{pit74} (see also~\cite{CC92}), this minimax principle gives rise to a union of closed geodesics (counted with multiplicity) of total length~$\dias(M)$ when the metric $g$ is smooth without conical singularities. 

Recall the following result (see \cite{BS07}).
\begin{theorem} \label{theo:dias}
There exists a constant $C$
such that every closed surface~$M$ of genus~$k$ endowed with a  Riemannian smooth metric $g$ satisfies
\begin{equation} \label{eq:A}
\area(M,g) \geq \frac{C}{k+1} \, \dias(M,g)^2. 
\end{equation}
The constant $C$ can be taken equal to~$10^{-16}$.
\end{theorem}

The best constant involved in inequality (\ref{eq:A}) is not known for any closed surface. Nevertheless the asymptotic behaviour of the above constant for large genus can not be improved, see \cite{BS07}. The inequality (\ref{eq:A}) remains valid for Riemannian metric of class $C^1$  with a finite number of conical singularities, as these metrics can be obtained as uniform limit (as metric spaces) of Riemannian smooth metrics and both the area and the diastole are continuous for this topology (see \cite[p.224 \& 269]{AZ67}).

\subsection{The singular metric and the Riemann uniformization theorem}\label{unif}

By the Riemann uniformization theorem there exists only one
conformal structure on $S^2$ (up to diffeomorphism). So we set
$S^2=\C \cup \{\infty\}$. Following \cite{Tro86} (see also \cite{Re01}), $g_c$ can be
written as
$$
\left(|z-a_1|\cdot |z-a_2| \cdot |z-a_3|\right)^{-4/3}|dz|^2
$$
up to a scale transformation where $(a_1,a_2,a_3)$ is any triple of
pairewise distinct points of $S^2$. We fix in the sequel the triple
to be $(-1,0,1)$.

 Let
$$
\M_{g_c}:=\left\{e^{2u}\cdot g_c \mid u \in
C^1(S^2,\R)\right\}
$$
be the space of metrics of class $C^1$ with three conical singularities of angle $2\pi/3$ on $S^2$ conformal to $g_c$.

\begin{proposition}\label{prop}
The infimum of the diastolic area over $\M_{g_c}$ equals the infimum of the diastolic area over the space of smooth metrics.
\end{proposition}

\begin{proof} By the Riemann uniformization theorem, every smooth Riemannian metric $g$ is isometric to a metric of the type $e^{2v} \cdot g_0$ where $g_0$ denotes the round metric and $v \in C^\infty(S^2,\R)$. As 
$$
g_0 = \left(\frac{2}{1+|z|^2}\right)^2 |dz|^2,
$$
we see that $g$ is homothetic (up to diffeomorphism) to the metric 
$$
e^{2v} \cdot \left(\frac{2}{1+|z|^2}\right)^2 \cdot \left(|z+1|\cdot |z| \cdot |z-1|\right)^{4/3} \cdot g_c
$$
which differs from $g_c$ by a conformal factor which lies in $C^1(S^2,\R_+)$. So we can approximate the smooth Riemannian metric $g$ by elements of $\M_{g_c}$ with respect to the $C^1$-topology. As the diastolic area is continuous with respect to this topology (the diastolic area is in fact continuous for the uniform topology on metric spaces), we get that the infimum of the diastolic area over $\M_{g_c}$ is less than the infimum of the diastolic area over the space of smooth metrics. The reverse inequality is obtained by approximating elements of $\M_{g_c}$ by Riemannian smooth metrics for the uniform topology on metric spaces.
\end{proof}

\subsection{The ramified cover and the isosystolic inequality on the torus}

Let 
$$
f : \T^2 \rightarrow S^2
$$
be the covering map of degree $3$ ramified over the three points
$\{-1,0,1\}$. If $g$ is in $\M_{g_c}$, we denote by $\tilde{g}$ the pull-back metric of $g$ by
the map $f$. The metric $\tilde{g}$ is $C^1$ everywhere. In fact,
near the ramified points, the map looks like the map $z \mapsto
z^3$. So each conical singularity of angle $2\pi/3$ lift to a conical singularity of angle $2\pi$.
Observe that
$$
\area(\T^2,\tilde{g})=3\cdot \area(S^2,g).
$$
The flat metric $\tilde{g}_c$ given by the hexagonal lattice and appearing as the pullback metric $f^\ast g_c$ is very special. First recall that for a closed Riemannian surface the {\it homotopy systole} is defined as the least length of a non-contractible closed curve and is denoted by $\sys_\pi$. Loewner's theorem (unpublished, see \cite{Pu52}) states the first known isosystolic inequality:

\begin{theorem}
For all metric $g$ of class $C^1$ on $\T^2$,
$$
\area(\T^2,g) \geq \frac{\sqrt{3}}{2} \, \sys_\pi(\T^2,g)^2 ,
$$
with equality if and only if $(\T^2,g)$ is homothetic to the flat
torus corresponding to the hexagonal lattice $(\T^2,\tilde{g}_c)$.
\end{theorem}

The study of such isosystolic inequalities is well developped, see \cite[p.104]{Be00} for instance.

\medskip

The systolic geometries of $(S^2,g_c)$  and $(\T^2,\tilde{g}_c)$ are intimately related : to each systole of $(\T^2,\tilde{g}_c)$ (that is a closed geodesic
realizing the homotopy systole) corresponds a geodesic loop of $(S^2,g_c)$ whose length equals the general systole (see
Figure \ref{fig:sys}). More precisely,
\begin{itemize}
\item Either the systole of $(\T^2,\tilde{g}_c)$ passes trough a
point of ramification, and its image by $f$ is a simple geodesic loop with base point a conical singularity,
\item Or the systole of $(\T^2,\tilde{g}_c)$ never goes trough a
point of ramification, and its image is a figure eight geodesic
avoiding conical singularities.
\end{itemize}

\begin{figure}[h]
\leavevmode \SetLabels
\L(.23*.2) $(1,0)$\\
\L(.52*.5) $f$\\
\L(.08*0.65) $(1/2,\sqrt{3}/2)$\\
\endSetLabels
\begin{center}
\AffixLabels{\centerline{\epsfig{file=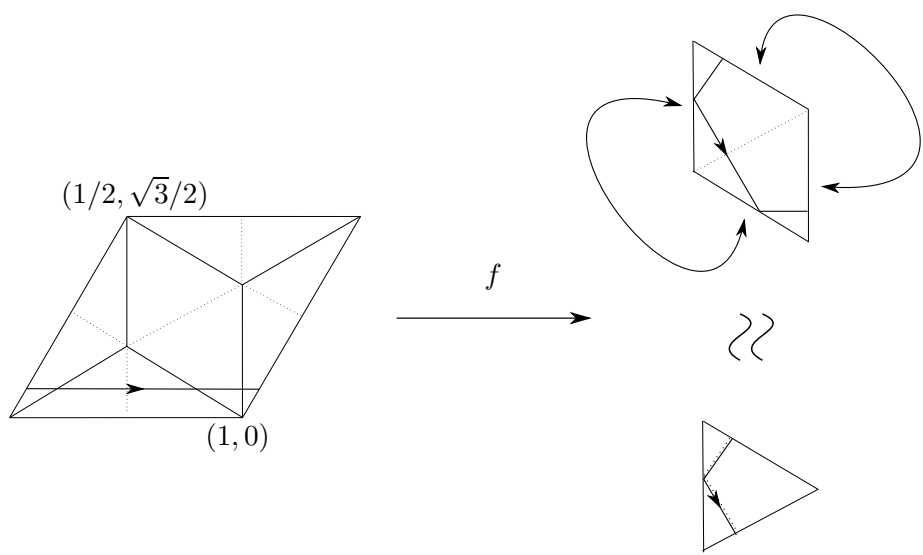,width=12cm,angle=0}}}
\end{center}
\caption{Projection of the systoles} \label{fig:sys}
\end{figure}

The systoles of $(\T^2,\tilde{g}_c)$ can be classified into three
one-parameter families. If we think at $(\T^2,\tilde{g}_c)$ as the
quotient of $\R^2$ by the hexagonal lattice spanned by the vectors
$(1,0)$ and $(1/2,\sqrt{3}/2)$, the horizontal lines, the lines parallel to the vector $(1/2,\sqrt{3}/2)$ and the lines parallel to the vector $(-1/2,\sqrt{3}/2)$
correspond to these three families of systoles. Observe now that
each of these families projects onto the same one-parameter family of geodesic loops of
$(S^2,g_c)$. We denote by $\gamma_s(t)=f(t,s)$ this family where
$(t,s) \in [0,1]\times[0,\sqrt{3}/2]$.

\section{Proof of the local diastolic inequality}

Let $g=e^{2u}\cdot g_c$ be a metric in $\M_{g_c}$ where $u \in
C^1(S^2,\R)$. By application of the
Cauchy-Schwarz's inequality and as $dv_{\tilde{g}_c}=dt \, ds$,
\begin{eqnarray}
\nonumber   \int_0^{\sqrt{3}/2} l_{g}(\gamma_s(t)) ds & = & \int_0^{\sqrt{3}/2} \left( \int_0^1 e^{u\circ f(t,s)} dt \right)ds \\
\nonumber & = &  \int_{\T^2} e^{u \circ f} dv_{\tilde{g}_c} \\
\nonumber & = & 3 \cdot \int_{S^2} e^u dv_{g_c} \\
\nonumber &  \leq & 3  \left( \int_{S^2} e^{2u} dv_{g_c}\right)^{1/2} \left( \int_{S^2} dv_{g_c}\right)^{1/2} \\
\nonumber &  \leq & 3 \left( \area(S^2,g)\right)^{1/2} \left( \area(S^2,g_c)\right)^{1/2}. \\
\nonumber
\end{eqnarray}

So
$$
\frac{\sqrt{3}}{2}\inf_{s \in [0,\sqrt{3}/2]} l_g(\gamma_s) \leq 3 \left( \area(S^2,g)\right)^{1/2} \left( \area(S^2,g_c)\right)^{1/2},
$$
which gives that
$$
\left(\inf_{s \in [0,\sqrt{3}/2]} l_g(\gamma_s)\right)^2 \leq 2\sqrt{3} \cdot \area(S^2,g)
$$
as $\area(S^2,g_c)=\frac{1}{2\sqrt{3}}$.

\begin{lemma}
There exists an open neighborhood $\Op$ of $g_c$ in $\M_{g_c}$ such that for all $g \in \Op$,
$$
\dias(S^2,g) \leq \inf_{s \in [0,\sqrt{3}/2]} l_g(\gamma_s).
$$
\end{lemma}

\begin{proof}[Proof of the Lemma]
To each geodesic loop $\gamma_s$ of $(S^2,g_c)$ we
associate a one-parameter family of one-cycles $\{z_s^\alpha\}$
where $\alpha \in [0,1]$ such that
\begin{itemize}
\item $z_s^{1/2}=\gamma_s$,
\item $\{z_s^\alpha\}$ starts and ends at one-cycles made of one or two points,
\item $\{z_s^\alpha\}$ induces a generator of $\pi_1{(\mathcal
Z}_1(S^2;\Z),\{0\})\simeq \Z$,
\item each $z_s^\alpha$ is made of one or two closed curves and have
length
$$
l_{g_c}(z_s^\alpha)=1-2\cdot|\alpha-1/2|
$$
where $|\cdot|$ denote the absolute value.
\end{itemize}

For this we consider two cases.\\

\noindent{\bf First case}. If $s=0,\frac{1}{2\sqrt{3}}$ or
$\frac{1}{\sqrt{3}}$, then the geodesic loop $\gamma_s$ goes through a
single singularity. Thus $\gamma_s$ bounds two disks $D_1$ and $D_2$
each one containing in its interior a conical singularity. For
$i=1,2$ we homotope $\gamma_s$ in $D_i$ to the singularity lying in
its interior through a $C^0$-family $\{z_{i,s}^\beta\}_{\beta \in
[0,1]}$ of geodesic loops (whose base point lies on the edge of the
triangle) of length
$$
l_{g_c}(z_{i,s}^\beta)=1-\beta.
$$
Then we set
$$
z_s^\alpha = \left\{\begin{array}{c}
z_{1,s}^{1-2\alpha} \text{ if } \alpha \in [0,\frac{1}{2}]\\
z_{2,s}^{2\alpha-1} \text{ if } \alpha \in [\frac{1}{2},1]\\
\end{array}\right. ,
$$
see Figure \ref{fig:family}, First case.
\\

\noindent{\bf Second case}. If $s=\frac{k}{2\sqrt{3}}+s'$ where
$k=0,1,2$ and $s' \in ]0,\frac{1}{2\sqrt{3}}[$, then 
$\gamma_s$ consists of a figure eight geodesic which avoids
singularities. So $\gamma_s$ decomposes into the concatenation of
two simple closed geodesic loops $\gamma_{1,s}$ and $\gamma_{2,s}$.
Let $D_i$ be the disk bounding by $\gamma_{i,s}$ and containing a
single singularity for $i=1,2$. Each $\gamma_{i,s}$ can be homotoped
in $D_i$ to the singularity lying in its interior through a family
$\{z_{i,s}^\beta\}_{\beta \in [0,1]}$ of geodesic loops (whose base point lies on the edge of the
triangle) of length
$$
l_{g_c}(z_{i,s}^\beta)=(1-\beta)l_{g_c}(\gamma_{i,s}).
$$
In a similar way, the piecewise geodesic $\tilde{\gamma}_s$ obtained by the concatenation of $\gamma_{1,s}$ and $\gamma_{2,s}^{-1}$ bounds an open disk containing a single singularity in its interior, and we denote by $D$ its closure. Again $\tilde{\gamma}_s$ can be homotoped in $D$  to the singularity lying in the interior of $D$ through a family $\{z_s^\beta\}_{\beta \in [0,1]}$ of piecewise geodesics of length
$$
l_{g_c}(z_s^\beta)=(1-\beta).
$$
Each piecewise geodesic $z_s^\beta$ consists of two geodesic arcs $z_{1,s}^\beta$ and $z_{2,s}^\beta$.

Then we set
$$
z_s^\alpha = \left\{\begin{array}{c}
z_{1,s}^{1-2\alpha}+z_{2,s}^{1-2\alpha} \text{ if } \alpha \in [0,\frac{1}{2}]\\
z_s^{2\alpha-1} \text{ if } \alpha \in [\frac{1}{2},1]\\
\end{array}\right. ,
$$
see  Figure \ref{fig:family},
Second case.\\

\begin{figure}[h]
\leavevmode \SetLabels
\L(.0*1) First case\\
\L(.49*1) Second case\\
\L(.25*.6) $\gamma_s$\\
\L(.72*.6) $\gamma_s$\\
\endSetLabels
\begin{center}
\AffixLabels{\centerline{\epsfig{file=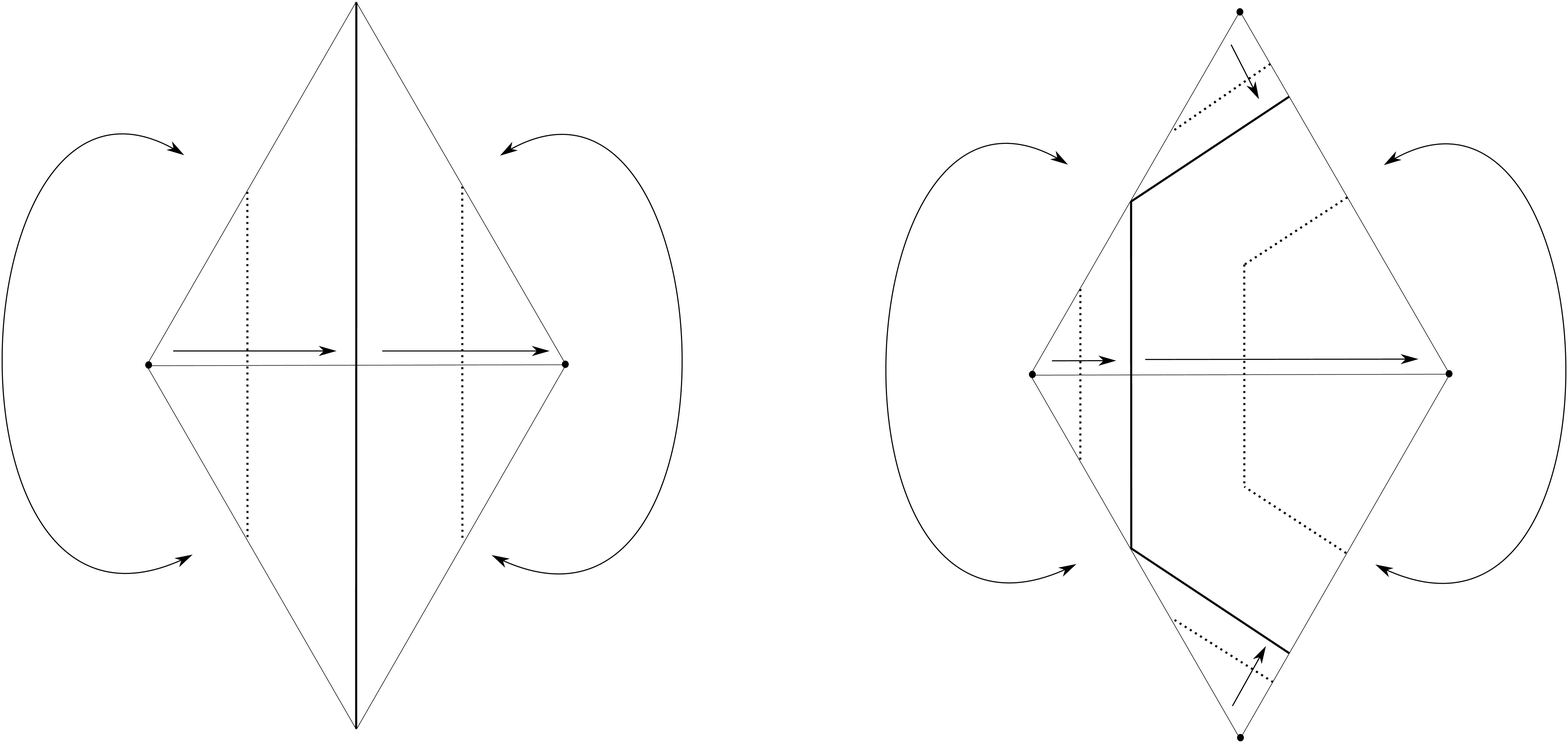,width=12cm,angle=0}}}
\end{center}
\caption{Families of one-cycles.} \label{fig:family}
\end{figure}

Denote by $\|\cdot\|_\infty$ the uniform norm on $C^1(S^2,\R)$. For each $\alpha \in [0,1/2[\cup]1/2,1]$, the family of $1$-cycles 
$$
\{z_s^{t/2+(1-t)\alpha}\mid t \in [0,1] \}
$$
 covers a domain  $ \Omega_s^\alpha$ whose boundary is the reunion of $\gamma_s$ and $z^\alpha_s$.  We can lift this domain to a domain $\tilde{\Omega}_s^\alpha$ of $\T^2$ as in Figure \ref{fig:domains}  which decomposes into two domains $\tilde{\Omega}_{1,s}^\alpha$ and $\tilde{\Omega}_{2,s}^\alpha$. One of this domain is reduced to a point if $s=0,1/(2\sqrt{3})$ or
$1/\sqrt{3}$. Each domain $\tilde{\Omega}_{i,s}^\alpha$ is bounded by three curves:  $\tilde{\gamma}_{i,s}$, $\tilde{z}^\alpha_{i,s}$ and $\tilde{c}_{i,s}$. Note that $\tilde{c}_{i,s}$ is composed of two connected components. The curve $\tilde{c}_{i,s}$ is oriented such that the concatenation of the first connected component of $\tilde{c}_{i,s}$, $\tilde{z}^\alpha_{i,s}$ and the second component of $\tilde{c}_{i,s}$ makes sense and is an arc homotopic to $\tilde{\gamma}_{i,s}$. We endow $\tilde{\Omega}_{i,s}^\alpha$ with the orientation such that its boundary is the concatenation of $\tilde{\gamma}_{i,s}$, the first connected component of $-\tilde{c}_{i,s}$, $-\tilde{z}^\alpha_{i,s}$ and the second component of $-\tilde{c}_{i,s}$. We can easily compute that
$$
\area(\tilde{\Omega}_{i,s}^\alpha,\tilde{g}_c)={1\over 4} \, \tan {\pi \over 6} \, (l_{g_c} (\gamma_{i,s})^2 -l_{g_c}(z^\alpha_{i,s})^2).
$$
Observe that  
$$
\area(\tilde{\Omega}_{i,s}^\alpha,\tilde{g}_c)\leq{1\over 2} \, \tan {\pi \over 6} \, (l_{g_c} (\gamma_{i,s}) -l_{g_c}(z^\alpha_{i,s}))
$$
as $l_{g_c} (\gamma_{i,s}) +l_{g_c}(z^\alpha_{i,s})\leq 2$.

\begin{figure}[h]
\leavevmode \SetLabels
\L(.0*1) Case $\alpha <\frac{1}{2}$ \\
\L(.52*1) Case $\alpha > \frac{1}{2}$\\
\L(.23*.68) $\tilde{\gamma}_s$\\
\L(.75*.46) $\tilde{\gamma}_s$\\
\L(.08*.3) $\tilde{z}^\alpha_{1,s}$\\
\L(.24*.86) $\tilde{z}^\alpha_{2,s}$\\
\L(.58*.8) $\tilde{z}^\alpha_{1,s}$\\
\L(.9*.15) $\tilde{z}^\alpha_{2,s}$\\
\L(.01*.62)  $\tilde{\Omega}_{1,s}^\alpha$\\
\L(.31*.84)  $\tilde{\Omega}_{2,s}^\alpha$\\
\L(.53*.62)  $\tilde{\Omega}_{1,s}^\alpha$\\
\L(.83*.84)  $\tilde{\Omega}_{2,s}^\alpha$\\
\endSetLabels
\begin{center}
\AffixLabels{\centerline{\epsfig{file =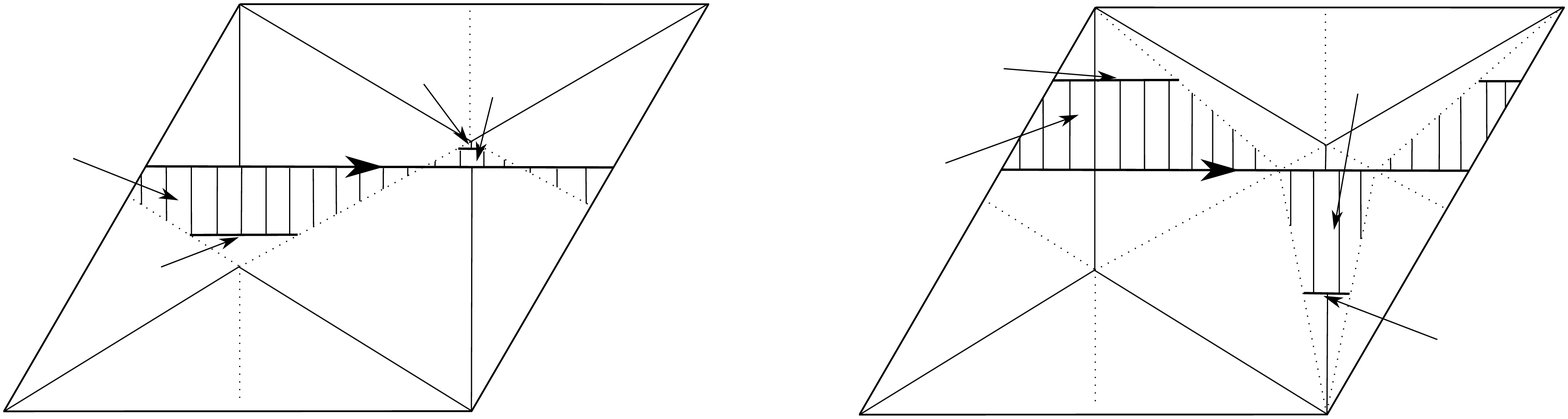,width=12cm,angle=0}}}
\end{center}
\caption{Domains.}
\label{fig:domains}
\end{figure}

For all $g=e^{2u} \cdot g_c
\in \M_{g_c}$ and by Stokes' Theorem,

\begin{eqnarray}
\nonumber l_g(\gamma_{i,s})-l_g(z^\alpha_{i,s})&=&  \int_{\tilde{\gamma}_{i,s}}e^{u\circ f} \, dt -\int_{\tilde{z}^\alpha_{i,s}}e^{u\circ f} \, dt\\
\nonumber &=&  \int_{\tilde{c}_{i,s}}e^{u\circ f} \, dt+\int_{\tilde{\Omega}^\alpha_{i,s}} d e^{u\circ f} \wedge dt\\
\nonumber &=& \int_{\tilde{c}_{i,s}}  \, dt+ \int_{\tilde{c}_{i,s}} (e^{u\circ f}-1) \, dt+\int_{\tilde{\Omega}^\alpha_{i,s}} {\partial \over \partial s} e^{u\circ f} \, ds \wedge dt\\
\nonumber &\geq& \int_{\tilde{c}_{i,s}}  \, dt- \int_{\tilde{c}_{i,s}} \|e^u-1\|_\infty \, dt-\sup_{\tilde{\Omega}^\alpha_{i,s}} \left| {\partial \over \partial s}  e^{u\circ f}\right| \cdot \area(\tilde{\Omega}_{i,s}^\alpha,\tilde{g}_c) \\
\nonumber &\geq&(1-\|e^u-1\|_\infty-{1 \over 2} \cdot \tan {\pi \over 6} \cdot  \sup_{\tilde{\Omega}^\alpha_{i,s}} \left| {\partial \over \partial s}  e^{u\circ f}\right|) \, (l_{g_c}(\gamma_{i,s})-l_{g_c}(z^\alpha_{i,s}))\\
\nonumber
\end{eqnarray}

Observe that there exists a positive constant $B$ such that
$$
\sup_{\tilde{\Omega}^\alpha_{i,s}} \left| {\partial \over \partial s}  e^{u\circ f}\right| \leq B \cdot \|de^u\|_0
$$
where $\|de^u\|_0=\sup \{|d {e^u}_x(v)| \mid (x,v) \in T S^2 \text{ with } g_0(v,v)=1\} $ is the uniform norm on closed one-forms on $S^2$ associated to the round metric  $g_0$.
 In fact,
$$
\left| {\partial \over \partial s}  e^{u\circ f} \right| (x)  =  \left|{d e^u}_{f(x)} \circ df_x ({\partial \over \partial s} )\right|\\
$$
for all $x \in \T^2$.
But $g_0(df_x ({\partial \over \partial s}),df_x ({\partial \over \partial s}))$ goes to $0$ when $x$ goes to a point of ramification. This is clear from the fact that $g_c(df_x ({\partial \over \partial s}),df_x ({\partial \over \partial s}))=1$ outside singularities and from the following expression of $g_0$: 
$$
g_0= \left(\frac{2}{1+|z|^2}\right)^2 \cdot \left(|z+1|\cdot |z| \cdot |z-1|\right)^{4/3} \cdot g_c
$$ 
where $z \in S^2=\C\cup\{\infty\}$.
 So we can set $B=\sup_{x \in \T^2}\sqrt{g_0(df_x ({\partial \over \partial s}),df_x ({\partial \over \partial s}))}$.

\medskip

Now define the set $\Op$ as the reunion of the open sets
$$
\{e^{2u}\cdot g_c \in \M_{g_c} \text{ such that } \|e^u-1\|_\infty< t \, \text{ and } \,  {1 \over 2} \tan {\pi \over 6} \,  B \, \|de^u\|_0<1-t\}
$$
for $t \in ]0,1[$. The set $\Op$ is an open neighborhood of $g_c$ in the $C^1$-topology and we get 
$$
l_g(\gamma_s)\geq l_g(z^\alpha_s)
$$
for all $g \in
\Op$, and for all $(s,\alpha)$. 

Fix $s' \in [0,\frac{\sqrt{3}}{2}]$ such that $l_g(\gamma_{s'})=\inf_{s \in [0,\sqrt{3}/2]} l_g(\gamma_s)$.
For all $g \in \Op$, we obtain 
\begin{eqnarray}
\nonumber \dias(S^2,g)& \leq &\sup_{\alpha \in [0,1]} l_g(z_{s'}^\alpha)\\
\nonumber &\leq &  l_g(\gamma_{s'})\\
\nonumber &\leq &\inf_{s \in [0,\sqrt{3}/2]} l_g(\gamma_s).\\
\nonumber
\end{eqnarray}

\end{proof}

Let us finish the proof of the Theorem \ref{th:dias-sing}. For all
$g \in \Op$, we have
$$
(\dias(S^2,g))^2 \leq  2\sqrt{3} \cdot \area(S^2,g).
$$
Now observe that the equality case imposes to the conformal factor $e^{2u}$ to be constant. So $g$ is homothetic to $g_c$.

\end{document}